\documentclass[11pt,reqno]{amsart}
\usepackage[T1]{fontenc}
\usepackage[utf8]{inputenc}
\usepackage{lmodern}
\usepackage{amsmath,amssymb,mathtools}
\usepackage[a4paper,margin=29mm]{geometry}
\usepackage{microtype}
\usepackage{needspace}
\usepackage{etoolbox}
\usepackage{xcolor}
\usepackage[colorlinks=true,linkcolor=blue!40!black,citecolor=blue!40!black,
urlcolor=blue!40!black,bookmarksnumbered=true]{hyperref}
\hypersetup{pdftitle={Additive preservers of mutual strong numerical radius orthogonality},
pdfauthor={Yulong Tian and Jinli Xu}}
\newtheorem{theorem}{Theorem}[section]
\newtheorem{lemma}[theorem]{Lemma}
\newtheorem{proposition}[theorem]{Proposition}

\theoremstyle{remark}
\newtheorem{remark}[theorem]{Remark}
\newtheorem{example}[theorem]{Example}
\numberwithin{equation}{section}
\newcommand{\C}{\mathbb C}
\newcommand{\R}{\mathbb R}
\newcommand{\BH}{\mathcal B(H)}
\newcommand{\BSH}{\mathcal B_s(H)}

\newcommand{\sorth}{\mathrel{\perp}_{w}^{\,s}}
\newcommand{\morth}{\mathrel{\perp\!\!\!\perp}_{w}^{\,s}}
\DeclareMathOperator{\ran}{ran}
\DeclareMathOperator{\rank}{rank}
\DeclareMathOperator{\Span}{span}
\DeclareMathOperator{\tr}{tr}
\allowdisplaybreaks[1]

\makeatletter
\patchcmd{\@maketitle}{\global\topskip42\p@\relax}
  {\global\topskip28\p@\relax}{}
  {\PackageWarning{paper-layout}{Title spacing adjustment failed}}
\patchcmd{\@maketitle}{\dimen@34\p@}{\dimen@22\p@}{}
  {\PackageWarning{paper-layout}{Abstract spacing adjustment failed}}
\makeatother

\title[Additive numerical radius orthogonality preservers]
{Additive preservers of mutual strong numerical radius orthogonality}
\author[Y. Tian]{Yulong Tian$^{*}$}
\address[Yulong Tian]{Department of Mathematics, Northeast Forestry University, Harbin, P. R. China}
\email{tylmath@nefu.edu.cn}
\thanks{$^{*}$Corresponding author.}
\author[J. Xu]{Jinli Xu}
\address[Jinli Xu]{Department of Mathematics, Northeast Forestry University, Harbin, P. R. China}
\email{jclixv@qq.com}
\date{}

\begin{document}
\begin{abstract}
Let $H$ be a complex Hilbert space with $\dim H\ge2$.
We characterize the nonzero additive maps on $\BH$ and on its self-adjoint
part $\BSH$ that preserve mutual strong numerical-radius orthogonality
in one direction, assuming surjectivity only in infinite dimension.
These maps are precisely those of the form $A\mapsto cUAU^*$, where $U$
is unitary or antiunitary, and $c$ is a nonzero complex scalar for $\BH$
and a nonzero real scalar for $\BSH$.
\end{abstract}
\maketitle

\section{Introduction and main results}\label{sec:introduction}

In a real or complex normed space, a vector $x$ is said to be
Birkhoff--James orthogonal to $y$ if $\|x+\lambda y\|\ge\|x\|$ for every
scalar $\lambda$. Classical results of Koldobsky~\cite{K1993} and Blanco and
Turn\v{s}ek~\cite{BT2006} show that every nonzero linear map preserving
this relation is a scalar multiple of an isometry.

Let $H$ be a complex Hilbert space, and let $\BH$ denote the algebra of
bounded linear operators on $H$. Replacing scalar multiples $\lambda B$
by arbitrary right multiples $BX$ yields strong Birkhoff--James
orthogonality on $\BH$, defined by $\|A+BX\|\ge\|A\|$ for every $X\in\BH$.
Graph and linear preserver problems for its mutual version, which
requires orthogonality in both directions, were investigated by
Aramba\v{s}i\'c, Guterman, Kuzma, Raji\'c, and Zhilina
in~\cite{AGKRZ2020,AGKRZ2021}, and Kuzma and
Stefanovi\'c~\cite{KS2026} subsequently extended the study to additive
preservers on finite-dimensional $C^*$-algebras.

Alongside the operator norm, another fundamental norm on $\BH$ is the
numerical radius, given by
\[
 w(A)=\sup_{\|x\|=1}|\langle x,Ax\rangle|.
\]
The Birkhoff--James orthogonality induced by this norm is called
numerical-radius orthogonality; see~\cite{ZW2020}.
Analogously, an operator $A$ is strongly numerical-radius orthogonal to $B$,
written $A\sorth B$, if
\begin{equation}\label{eq:strong-definition}
 w(A+BX)\ge w(A)\qquad\text{for every }X\in\BH,
\end{equation}
and we write $A\morth B$ when both $A\sorth B$ and $B\sorth A$ hold.
Sun, Hou, and Qi~\cite{SHQ2026} studied the graphs associated with
mutual numerical-radius orthogonality and its strong version.
These results motivate the corresponding preserver problem.

We characterize nonzero additive maps preserving mutual strong
numerical-radius orthogonality in one direction on $\BH$ and on its
self-adjoint part, assuming surjectivity only in infinite dimension.
Our first result concerns the full algebra $\BH$.

\par\medskip\noindent
\begin{minipage}{\linewidth}
\begin{theorem}\label{thm:full}
Let $H$ be a complex Hilbert space with $\dim H\ge2$, and let
$\Phi:\BH\to\BH$ be a nonzero additive map. If $H$ is infinite
dimensional, assume in addition that $\Phi$ is surjective. Then
\begin{equation}\label{eq:full-preservation}
 A\morth B\quad\Longrightarrow\quad\Phi(A)\morth\Phi(B)
 \qquad(A,B\in\BH)
\end{equation}
if and only if there exist $c\in\C\setminus\{0\}$ and a unitary or
antiunitary operator $U$ on $H$ such that
\begin{equation}\label{eq:full-standard}
 \Phi(A)=cUAU^*\qquad(A\in\BH).
\end{equation}
\end{theorem}
\end{minipage}\par\medskip

\Needspace{8\baselineskip}
We also consider the real vector space
$\BSH=\{A\in\BH:A=A^*\}$ of self-adjoint operators.
The relations $\sorth$ and $\morth$ are inherited from $\BH$;
in particular, $X$ in \eqref{eq:strong-definition} still ranges
over all of $\BH$.
Our second result gives the corresponding classification in this
setting.

\begin{theorem}\label{thm:selfadjoint}
Let $H$ be a complex Hilbert space with $\dim H\ge2$, and let
$\Phi:\BSH\to\BSH$ be a nonzero additive map. If $H$ is infinite
dimensional, assume in addition that $\Phi$ is surjective. Then
\begin{equation}\label{eq:selfadjoint-preservation}
 A\morth B\quad\Longrightarrow\quad\Phi(A)\morth\Phi(B)
 \qquad(A,B\in\BSH)
\end{equation}
if and only if there exist $c\in\R\setminus\{0\}$ and a unitary or
antiunitary operator $U$ on $H$ such that
\begin{equation}\label{eq:selfadjoint-standard}
 \Phi(A)=cUAU^*\qquad(A\in\BSH).
\end{equation}
\end{theorem}

We prove Theorem~\ref{thm:full} separately in finite and infinite
dimensions (Sections~\ref{sec:finite} and~\ref{sec:infinite}).
In finite dimension, we first establish real linearity and show that
$\ker\Phi$ contains no nonzero singular matrix. This yields preservation
of singularity. We then prove that $\ker\Phi=\{0\}$---and hence automatic
bijectivity---by combining a dimension bound for spaces of singular
matrices with an analysis of the subspaces spanned by operators mutually
strongly orthogonal to a singular matrix. This permits us to
apply the additive singularity-preserver theorem of Fo\v{s}ner and
\v{S}emrl~\cite{FS2005}, as formulated
in~\cite[Theorem~1.4]{KS2026}. Examining the action of $\Phi$ on rank-one
nilpotents and orthogonal projections then reduces the resulting
representations to~\eqref{eq:full-standard}.

In infinite dimension, we first prove automatic continuity and real
linearity. An iterative argument with orthogonal finite-rank
perturbations then shows that $\Phi$ scales the numerical radius by a
constant factor on finite-rank operators. This scaling identity is
extended to all of $\BH$ through an upper bound via averaged compressions
and a lower bound via the directional derivative of the numerical radius
at rank-one projections. For surjective maps, the numerical-radius
distance-isometry theorem of Bai and Hou~\cite[Theorem~$1'$]{BH2004}
gives the desired classification after we exclude the adjoint alternative
involving $A^*$.

For the self-adjoint case (Section~\ref{sec:selfadjoint}), we establish
automatic continuity, real linearity, and injectivity, followed by
preservation of noninvertibility. Consequently, $\Phi$ is bijective
(using the assumed surjectivity in infinite dimension) and $\Phi^{-1}$
preserves invertibility, so we may apply the
invertibility-preserver theorem of Li, Rodman, and
\v{S}emrl~\cite[Theorem~6]{LRS2003} to $\Phi^{-1}$. Finally, preservation
of orthogonality between rank-one projections reduces the resulting
congruence representation to~\eqref{eq:selfadjoint-standard}.

\section{Numerical radius and rank-one operators}\label{sec:preliminaries}

The inner product is conjugate-linear in its first variable. For $x,y\in H$,
the operator $x\otimes y$ is defined by $(x\otimes y)z=x\langle y,z\rangle$.
If $x$ is a unit vector, put
$P_x=x\otimes x$. More generally, $P_L$ denotes the orthogonal projection onto a
closed subspace $L$.
A unit vector $m$ is a numerical-radius maximizing vector for $A$ if
$|\langle m,Am\rangle|=w(A)$; we call $\C m$ a maximizing ray for $A$.
We shall use the elementary identities and inequalities
\begin{equation}\label{eq:basic-radius}
 \tfrac12\|T\|\le w(T)\le\|T\|,\qquad
 w(S\oplus T)=\max\{w(S),w(T)\},\qquad
 w(P_LT|_L)\le w(T).
\end{equation}
An operator $A$ is supported on a closed subspace $L$ if $A=P_LAP_L$.
Operators with orthogonal supports satisfy $A\morth B$, by compression
to their respective supports.
Both orthogonality relations are unchanged when $A$ and $B$ are
independently multiplied by nonzero complex scalars.
Taking $X=-I$ in~\eqref{eq:strong-definition}, we also obtain
\begin{equation}\label{eq:self-orthogonality}
 A\sorth A\quad\Longleftrightarrow\quad A=0.
\end{equation}
\begin{lemma}\label{lem:standard}
Let $c\in\C\setminus\{0\}$, and let $U$ be unitary or antiunitary on $H$.
The map $T\mapsto cUTU^*$ is additive and bijective and preserves
$\sorth$ and $\morth$ in both directions. Moreover,
\[
 w(cUTU^*)=|c|w(T),\qquad \|cUTU^*\|=|c|\|T\|.
\]
\end{lemma}
This follows from invariance under unitary and antiunitary conjugation
and the change of variables $Y=U^*XU$ in~\eqref{eq:strong-definition}.

\Needspace{11\baselineskip}
\begin{lemma}[{\cite[Proposition~2.2]{SHQ2026}}]\label{lem:witness}
Let $A,B\in\BH$. If there are unit vectors $x_j$ such that
\[
 |\langle x_j,Ax_j\rangle|\longrightarrow w(A),\qquad \|B^*x_j\|\longrightarrow0,
\]
then $A\sorth B$. In particular, $|\langle x,Ax\rangle|=w(A)$ and $B^*x=0$ for a unit
vector $x$ imply $A\sorth B$.
\end{lemma}

\begin{proof}
For each fixed $X\in\BH$,
\[
 |\langle x_j,BXx_j\rangle|\le\|B^*x_j\|\,\|X\|\longrightarrow0.
\]
Hence $w(A+BX)\ge\lim_j|\langle x_j,(A+BX)x_j\rangle|=w(A)$.
\end{proof}

\begin{lemma}\label{lem:rankone}
Let $R=u\otimes v\ne0$ be a rank-one operator. Then
\begin{equation}\label{eq:rankone-radius}
 w(R)=\frac{\|u\|\|v\|+|\langle v,u\rangle|}{2}.
\end{equation}
If $\langle v,u\rangle\ne0$, there is a unique numerical-radius maximizing ray $\C m$,
where $\|m\|=1$, and
\begin{equation}\label{eq:rankone-kernel}
 R\sorth B\quad\Longleftrightarrow\quad B^*m=0.
\end{equation}
\end{lemma}
\begin{proof}
Formula~\eqref{eq:rankone-radius} and the corresponding maximizing
vectors are given in~\cite[pp.~359--360]{FK1993}.
After normalizing $u,v$ and adjusting a scalar phase, it is enough to
consider $R=p\otimes q$ with unit vectors $p,q$ and $\langle q,p\rangle=\rho\ge0$.
If $\rho>0$, the unique maximizing ray is spanned by
\[
 m=\frac{p+q}{\sqrt{2+2\rho}}.
\]

The sufficiency in~\eqref{eq:rankone-kernel} follows from
Lemma~\ref{lem:witness}. For necessity, let $\rho>0$ and suppose that
$R\sorth B$ but $B^*m\ne0$. Put
\[
 h=BB^*(p+q),\qquad
 X=-\bigl(B^*(p+q)\bigr)\otimes q.
\]
Then $R+tBX=(p-th)\otimes q$. By~\eqref{eq:rankone-radius}
and the inequality $2ab\le a^2+b^2$, for $t>0$ we have
\[
 \begin{aligned}
 2w(R+tBX)
 &=\|p-th\|+|\rho-t\langle q,h\rangle|\\
 &\le \frac{1+\|p-th\|^2}{2}
      +\frac{\rho^2+|\rho-t\langle q,h\rangle|^2}{2\rho}\\
 &=1+\rho-t\|B^*(p+q)\|^2\\
 &\quad+\frac{t^2}{2}
 \left(\|h\|^2+\frac{|\langle q,h\rangle|^2}{\rho}\right).
 \end{aligned}
\]
Since $B^*(p+q)\ne0$, the last expression is smaller than
$1+\rho=2w(R)$ for all sufficiently small $t>0$.
This contradicts $R\sorth B$.
\end{proof}

The next sufficient condition will be used for rank-one nilpotents.

\begin{lemma}\label{lem:nilpotent}
Let $x,y\in H$ be orthonormal. If $A\in\BH$ satisfies
\[
 A^*x=\lambda A^*y\qquad\text{for some }|\lambda|\le1,
\]
then $x\otimes y\sorth A$.
\end{lemma}
\begin{proof}
Put $N=x\otimes y$ and $M=\Span\{x,y\}$. For each $X\in\BH$, let
$C=P_MAX|_M$. Then $C^*x=\lambda C^*y$. If $C\ne0$,
\cite[Corollary~2.5]{SHQ2026} gives $N|_M\sorth C$, and hence
$w(N|_M+C)\ge1/2$; this inequality is immediate when $C=0$.
By compression,
\[
 w(N+AX)\ge w(N|_M+C)\ge\tfrac12=w(N).
\]
\end{proof}

\Needspace{10\baselineskip}
The angle estimate below appears in the proof
of~\cite[Proposition~4.1]{SHQ2026} in finite dimension.
The equality case can also be deduced from
\cite[Corollary~4, $n=1$]{GW2009}.

\begin{lemma}\label{lem:angle}
Let $0\ne A\in\BH$, and let $m,k$ be unit vectors such that
$|\langle m,Am\rangle|=w(A)$ and $A^*k=0$. Then
\begin{equation}\label{eq:angle}
 |\langle m,k\rangle|\le\frac1{\sqrt2}.
\end{equation}
If equality holds, there is a two-dimensional reducing subspace $M$ of
$A$ such that, relative to $H=M\oplus M^\perp$,
\begin{equation}\label{eq:equality-block}
 A=R\oplus D,\qquad \rank R=1,\quad R^2=0,\quad
 w(D)\le w(R)=w(A).
\end{equation}
\end{lemma}
\begin{proof}
Put $r=w(A)>0$ and $A_0=e^{-i\arg(\langle m,Am\rangle)}A=G+iK$, where $G,K$ are
self-adjoint. Since $\langle m,A_0m\rangle=r$ and $\|G\|,\|K\|\le r$, we have
$\langle m,Gm\rangle=r$ and $\langle m,Km\rangle=0$.
Hence $Gm=rm$ and $Km\perp m$.
Since $\langle m,Am\rangle\ne0$, the vector $k$ cannot be parallel to $m$.
After changing the phase of $k$, write
\[
 k=sm+te_2,\qquad s=|\langle m,k\rangle|\ge0,\quad t=\sqrt{1-s^2}>0,
 \quad e_2\perp m,\quad\|e_2\|=1.
\]
Since $A_0^*k=0$, the identity $\langle k,A_0m\rangle=0$ gives
\[
 0=sr+it\,\langle e_2,Km\rangle,\qquad sr\le t\|Km\|\le tr.
\]
This proves~\eqref{eq:angle}.

If $s=t=1/\sqrt2$, equality forces $Km=ir e_2$ and $\|Km\|=r$.
Since $r^2I-K^2\ge0$ and
$\langle m,(r^2I-K^2)m\rangle=0$, we have $K^2m=r^2m$,
and therefore $Ke_2=-ir m$.
Now $0=\sqrt2 A_0^*k=(G-iK)(m+e_2)=Ge_2+re_2$, so $Ge_2=-re_2$.
Consequently $\Span\{m,e_2\}$ reduces $A_0$, and its matrix there is
\[
 r\begin{pmatrix}1&1\\-1&-1\end{pmatrix}.
\]
This is a rank-one nilpotent with numerical radius $r$. Restoring the
scalar phase gives~\eqref{eq:equality-block}.
\end{proof}

\section{Neighborhood spans}\label{sec:spans}

For $A\in\BH$, set
\begin{equation}\label{eq:neighborhood-definition}
 \mathcal N(A)=\{B\in\BH:A\morth B\},\qquad
 \mathcal V(A)=\Span_{\C}\mathcal N(A).
\end{equation}
We call $\mathcal N(A)$ the mutual neighborhood of $A$.
All linear spans in this paper are algebraic spans.
Since $\mathcal N(A)$ is closed under complex scalar multiplication,
its real and complex linear spans coincide.

We next distinguish nilpotent and nonnilpotent rank-one operators by
the linear spans of their mutual neighborhoods.

\begin{lemma}\label{lem:span}
Let $\dim H\ge2$, and let $R\in\BH$ have rank one. Then
\[
 \mathcal V(R)=\BH\quad\Longleftrightarrow\quad R^2=0.
\]
\end{lemma}
\begin{proof}
If $R^2\ne0$, Lemma~\ref{lem:rankone} gives a unit vector $m$ such that
\[
 \mathcal N(R)\subseteq\{B:B^*m=0\}.
\]
The space on the right is a proper complex linear subspace of $\BH$.

If $R^2=0$, unitary conjugation and scalar homogeneity
reduce matters to $N=e_1\otimes e_2$ for orthonormal $e_1,e_2$.
Put $L=e_1^\perp$ and $P=P_L$.
If $D=PDP$, then $D^*e_1=0$, so for every $X\in\BH$,
\[
 w(N+DX)\ge\tfrac12\|(N+DX)^*e_1\|
 =\tfrac12=w(N).
\]
Thus $N\sorth D$. Also,
\[
 P(D+NX)P=D\qquad(X\in\BH),
\]
and compression gives $D\sorth N$.
For $x\in L$, put $S=x\otimes e_1$. Since $S^*e_1=0$, the same
estimate gives $N\sorth S$. Moreover,
\[
 w(S+NX)\ge\tfrac12\|P(S+NX)\|
 =\tfrac12\|x\|=w(S)\qquad(X\in\BH),
\]
so $S\sorth N$. Hence $x\otimes e_1\in\mathcal N(N)$.
It follows that
\begin{equation}\label{eq:range-hyperplane-span}
 \{T\in\BH:\ran T\subseteq L\}\subseteq\mathcal V(N),
\end{equation}
because every such $T$ is the sum of $PTP$ and $(Te_1)\otimes e_1$.

Since
\[
 T=PT+e_1\otimes(T^*e_1)\qquad(T\in\BH),
\]
it remains to prove that $e_1\otimes y\in\mathcal V(N)$ for every
$y\in H$. If $F\in\mathcal N(N)$, then
$PF\in\mathcal V(N)$ by~\eqref{eq:range-hyperplane-span}, and hence
\[
 P_{e_1}F=F-PF\in\mathcal V(N).
\]
Consider
\[
 F_0=(e_1+e_2)\otimes(e_1-e_2),\qquad
 F_1=(e_1+ie_2)\otimes(e_1-ie_2).
\]
Both are rank-one nilpotents of numerical radius $1$.
By Lemma~\ref{lem:witness}, the vector $e_2$ witnesses $F_j\sorth N$.
The vectors $(e_1-e_2)/\sqrt2$ and
$(e_1-ie_2)/\sqrt2$ witness $N\sorth F_0$ and $N\sorth F_1$,
respectively.
We have
\[
 P_{e_1}F_0=e_1\otimes(e_1-e_2),\qquad
 P_{e_1}F_1=e_1\otimes(e_1-ie_2),
\]
so taking linear combinations gives
$e_1\otimes e_1,e_1\otimes e_2\in\mathcal V(N)$.

If $z\perp e_1,e_2$ is a unit vector, put
\[
 F_z=(e_1+e_2+\sqrt2z)\otimes(e_1+e_2-\sqrt2z).
\]
It has numerical radius $2$. The vector $z$ witnesses $F_z\sorth N$,
and $(e_1-e_2)/\sqrt2$ witnesses $N\sorth F_z$.
Since $P_{e_1}F_z=e_1\otimes(e_1+e_2-\sqrt2z)\in\mathcal V(N)$,
we obtain $e_1\otimes z\in\mathcal V(N)$.
Thus $e_1\otimes y\in\mathcal V(N)$ for every $y\in H$, and
$\mathcal V(N)=\BH$ follows.
\end{proof}

\begin{lemma}\label{lem:idempotent}
Let $Q\in\BH$ be a rank-one idempotent. If $Q\sorth(I-Q)$, then $Q$
is an orthogonal projection.
\end{lemma}
\begin{proof}
Taking $X=I$ in the defining inequality gives $w(Q)\le1$.
Write $Q=u\otimes v$ with $\langle v,u\rangle=1$. By
\eqref{eq:rankone-radius},
\[
 1\ge w(Q)=\frac{\|u\|\|v\|+1}{2}\ge1.
\]
Equality in the Cauchy--Schwarz inequality forces $u,v$ to be parallel,
so $Q$ is an orthogonal projection.
\end{proof}

\section{The finite-dimensional case of Theorem~\texorpdfstring{\ref{thm:full}}{\getrefnumber{thm:full}}}\label{sec:finite}

Throughout this section, $H$ is finite dimensional. Fix an orthonormal
basis and identify $\BH$ with $M_n(\C)$, where $n=\dim H\ge2$.
Under this identification, the representation~\eqref{eq:full-standard}
is equivalent to
\begin{equation}\label{eq:finite-standard}
 \Phi(A)=cVAV^*\quad\text{or}\quad
 \Phi(A)=cV\overline A\,V^*\qquad(A\in M_n(\C)),
\end{equation}
where $V\in M_n(\C)$ is unitary and $\overline A$ denotes entrywise
conjugation. We first prove that every nonzero additive preserver in
this setting is automatically bijective.

\begin{lemma}\label{lem:finite-real-linearity}
Let $\mathcal X$ be either $M_n(\C)$ or $\mathcal B_s(\C^n)$.
Every additive map $\Phi:\mathcal X\to M_n(\C)$ preserving $\morth$
in one direction is continuous and real linear.
\end{lemma}
\begin{proof}
First suppose that $\mathcal X=\mathcal B_s(\C^n)$, and put
$q(A)=w(\Phi(A))$. For every rank-one orthogonal projection $P$,
the relation $P\morth I-P$ gives $q(P)\le q(I)$.
Fix orthonormal vectors $x,y$ and put $P=P_x$, $Q=P_y$.
For $|t|\le1$, let $a=\sqrt{(1+t)/2}$ and $b=\sqrt{(1-t)/2}$.
Since
\[
 t(P-Q)=P_{ax+by}-P_{bx+ay},
\]
subadditivity gives $q(t(P-Q))\le2q(I)$.
Also $tP\morth tQ$, so
\[
 q(tP)\le q(t(P-Q))\le2q(I)\qquad(|t|\le1).
\]
Thus the additive function $t\mapsto\Phi(tP)$ is bounded near zero,
and hence is continuous and real linear. Spectral decomposition now
shows that $\Phi$ is real linear on $\mathcal B_s(\C^n)$.

For a map defined on $M_n(\C)$, apply the preceding argument to
the maps $A\mapsto\Phi(A)$ and $A\mapsto\Phi(iA)$ on
$\mathcal B_s(\C^n)$. Both preserve $\morth$ in one direction.
Writing each matrix as $A=G+iL$ with $G,L$ self-adjoint proves real
linearity on $M_n(\C)$. Continuity follows from finite dimensionality.
\end{proof}

The following alternative for neighborhood spans will be used both to
prove preservation of singularity and to establish injectivity.

\begin{lemma}\label{lem:hyperplane-span}
Let $0\ne A\in M_n(\C)$ be singular, and let $m$ be a unit vector
such that $|\langle m,Am\rangle|=w(A)$. Put
\[
 \mathcal H_m=\{B\in M_n(\C):B^*m=0\}.
\]
Then
\[
 A\in\mathcal V(A)\qquad\text{or}\qquad
 \mathcal H_m\subseteq\mathcal V(A).
\]
\end{lemma}
\begin{proof}
Choose a unit vector $k\in\ker A^*$. By Lemma~\ref{lem:angle},
$|\langle m,k\rangle|\le1/\sqrt2$.
If equality holds, that lemma gives an orthogonal decomposition
\[
 A=R\oplus D,\qquad \rank R=1,\quad R^2=0,\quad
 w(D)\le w(R)=w(A),
\]
where $R$ acts on a two-dimensional subspace $M$.
If $B\in\mathcal B(M)$ satisfies $B\morth R$, then
$B\oplus0\morth A$ by compression to $M$.
Lemma~\ref{lem:span}, applied on $M$, therefore gives
$R\oplus0\in\mathcal V(A)$.
Since $R\morth R^*$ and
$w(R^*\oplus D)=w(R^*\oplus0)=w(R)$, the same compression argument
shows that $R^*\oplus D$ and $R^*\oplus0$ belong to $\mathcal N(A)$.
Their difference is $0\oplus D$, so $A\in\mathcal V(A)$.

Suppose now that $|\langle m,k\rangle|<1/\sqrt2$.
Put $L=m^\perp$, and let $\mathcal M(L)=P_LM_n(\C)P_L$
be the matrices supported on $L$.
Since $\mathcal H_m$ consists of the matrices with range in $L$,
for any $w\notin L$ we have
\[
 \mathcal H_m=P_LM_n(\C)
 =\mathcal M(L)+\{x\otimes w:x\in L\}.
\]
It is therefore enough to show that both summands are contained in
$\mathcal V(A)$, taking $w=k$ when $\langle m,k\rangle\ne0$
and $w=m$ otherwise.

First suppose that $\langle m,k\rangle\ne0$. The set
\[
 \Omega=\{u\in L:\|u\|=1,\ |\langle k,u\rangle|>1/\sqrt2\}
\]
is nonempty and relatively open in the unit sphere of $L$, since
$\|P_Lk\|>1/\sqrt2$.
For $u\in\Omega$, put
\[
 \alpha=\langle k,u\rangle,\qquad v=2\alpha k-u,\qquad
 z=k-\overline\alpha u,\qquad R=u\otimes v,\qquad N=u\otimes z.
\]
Here $1/\sqrt2<|\alpha|<1$, $\|v\|=1$, and
\[
 \langle v,u\rangle=2|\alpha|^2-1>0,
 \qquad \langle k,Rk\rangle=|\alpha|^2=w(R).
\]
Consequently $R\sorth A$ by Lemma~\ref{lem:witness}, while
$A\sorth R$ follows from $R^*m=0$.

Also $z\perp u$ and $\delta:=\|z\|=\sqrt{1-|\alpha|^2}>0$.
Since $A^*k=0$, we have
\[
 A^*u=-\frac{\delta}{\overline\alpha}A^*(z/\delta),
 \qquad \frac{\delta}{|\alpha|}<1.
\]
Lemma~\ref{lem:nilpotent} gives $N\sorth A$, and $N^*m=0$ gives
$A\sorth N$. Thus $R,N\in\mathcal N(A)$.
The identities
\[
 R=2\overline\alpha\,(u\otimes k)-u\otimes u,\qquad
 N=u\otimes k-\alpha(u\otimes u)
\]
form a linear system for $u\otimes k$ and $u\otimes u$ with
nonzero determinant $1-2|\alpha|^2$. Hence
\[
 u\otimes k,\ u\otimes u\in\mathcal V(A)\qquad(u\in\Omega).
\]
The vectors in $\Omega$ span $L$, and their projections $u\otimes u$
span $\mathcal M(L)$. Indeed, if a linear functional $f$ on
$\mathcal M(L)$ vanishes on these projections, then
$u\mapsto f(u\otimes u)$ is a polynomial in real coordinates that,
by homogeneity, vanishes on a nonempty open cone in $L$.
It therefore vanishes on all of $L$, and polarization gives $f=0$.
The decomposition with $w=k$ now gives
$\mathcal H_m\subseteq\mathcal V(A)$.

Now suppose that $\langle m,k\rangle=0$, so $k\in L$.
Let $J=2P_k-I_L$ on $L$, extended by zero on $L^\perp$.
For every unit $u\in L$ with $|\langle k,u\rangle|>1/\sqrt2$, put
$\alpha=\langle k,u\rangle$. Then
$R_u=P_uJ=u\otimes(2\alpha k-u)$ belongs to $\mathcal N(A)$
by the same calculation as for $R$ above.
Since the projections $P_u$ span $\mathcal M(L)$ and $J|_L$ is
invertible, we obtain $\mathcal M(L)\subseteq\mathcal V(A)$.

Since $k\perp m$ and $A^*k=0$,
Lemmas~\ref{lem:witness} and~\ref{lem:nilpotent} give
$k\otimes m\morth A$.
For $x\perp m,k$, the matrix $D=x\otimes m$ is supported on $k^\perp$.
Choose $\varepsilon>0$ such that $\varepsilon w(D)\le1$ and set
$B=P_k+\varepsilon D$. Relative to $\C^n=\C k\oplus k^\perp$,
$B=1\oplus\varepsilon D|_{k^\perp}$. Thus $w(B)=1$, attained at $k$,
and $B^*m=0$. Since $A^*k=0$, Lemma~\ref{lem:witness}, applied at
$k$ and $m$, gives $B\morth A$.
As $P_k\morth A$ as well, we obtain
$x\otimes m\in\mathcal V(A)$.
Together with $k\otimes m$, these matrices give all
$x\otimes m$ with $x\in L$.
The decomposition with $w=m$ therefore gives
$\mathcal H_m\subseteq\mathcal V(A)$.
\end{proof}

\begin{lemma}\label{lem:singularity}
Let $\Phi:M_n(\C)\to M_n(\C)$ be a nonzero additive map preserving
$\morth$ in one direction. Then $\ker\Phi$ contains no nonzero
singular matrix, and $\Phi$ maps singular matrices to singular matrices.
\end{lemma}
\begin{proof}
By Lemma~\ref{lem:finite-real-linearity}, $\Phi$ is continuous and
real linear. Put $K=\ker\Phi$.

\smallskip
\noindent\emph{Rank-one nilpotents.}
For unit vectors $u,v$ with $\langle u,v\rangle=-1/2$, put $r=u+v$.
Then $\|r\|=1$ and $\langle r,u\rangle=\langle r,v\rangle=1/2$.
By Lemma~\ref{lem:rankone}, the maximizing rays of $u\otimes r$ and
$v\otimes r$ are $\C(u+r)$ and $\C(v+r)$, respectively.
Since $\langle v,u+r\rangle=\langle u,v+r\rangle=0$, we obtain
\[
 u\otimes r\morth v\otimes r.
\]

Suppose that $K$ contains a nonzero rank-one nilpotent.
After dividing by its norm, write it as $N=x\otimes e$ with
$x,e$ orthonormal.
For $|\eta|=1$, put
\[
 r=\eta e,\qquad
 u=\frac{r+\sqrt3x}{2},\qquad
 v=\frac{r-\sqrt3x}{2}.
\]
These vectors satisfy the conditions above.
Since $u\otimes r=\overline\eta(u\otimes e)$, and similarly for $v$,
we also have $u\otimes e\morth v\otimes e$.
Moreover,
\[
 u\otimes e+v\otimes e=\eta P_e,\qquad
 u\otimes e-v\otimes e=\sqrt3N.
\]
Since $\sqrt3N\in K$, their images are equal and hence zero.
Taking their sum gives $\Phi(\eta P_e)=0$.
Varying $\eta$ and using real linearity yields $\C P_e\subseteq K$.
Now fix $\eta=1$ and replace $x$ by any unit vector $y\perp e$.
The sum of the resulting pair is $P_e\in K$, so its two images
are opposite and mutually strongly orthogonal, and therefore zero.
Their difference gives $y\otimes e\in K$.
Together with $\C P_e\subseteq K$, this proves
\[
 w\otimes e\in K\qquad(w\in\C^n).
\]

Fix a unit vector $f\perp e$, and put
$W=\Span_{\R}\{e,f\}$.
Consider the real-linear map
\[
 \Lambda:W\longrightarrow M_n(\C),\qquad \Lambda(u)=\Phi(u\otimes f).
\]
Choose a unit vector $u\in W$ such that
\[
 w(\Lambda(u))=M:=\max_{\substack{v\in W\\\|v\|=1}}w(\Lambda(v)).
\]
The two unit vectors $v\in W$ with $\langle u,v\rangle=-1/2$
give noncollinear vectors $u+v$.
We may therefore choose one for which
\[
 r=u+v=ae+bf,\qquad a,b\in\R,\quad b\ne0.
\]
Since $\Phi(w\otimes e)=0$ for every $w\in\C^n$, real linearity gives
\[
 \Phi(u\otimes r)=b\Lambda(u),\qquad
 \Phi(v\otimes r)=b\Lambda(v).
\]
Since $u\otimes r\morth v\otimes r$ and $b\ne0$,
preservation gives $\Lambda(u)\morth\Lambda(v)$.
Taking the right multiplier $I/2$ and using the definition of $M$,
we obtain
\[
 M\le w\bigl(\Lambda(u)+\tfrac12\Lambda(v)\bigr)
 =w\bigl(\Lambda(u+\tfrac12v)\bigr)
 \le\|u+\tfrac12v\|M=\frac{\sqrt3}{2}M.
\]
Thus $M=0$ and $\Lambda=0$, so $e\otimes f\in K$.
Applying the argument that yielded $w\otimes e\in K$ to this
rank-one nilpotent gives $w\otimes f\in K$ for every $w\in\C^n$.
Since the unit vector $f\perp e$ was arbitrary, complete $e$ to an
orthonormal basis $e_1=e,e_2,\ldots,e_n$.
Every matrix $A$ has the decomposition
\[
 A=\sum_{j=1}^n(Ae_j)\otimes e_j,
\]
whose summands all lie in $K$. Thus $\Phi=0$, a contradiction.

\smallskip
\noindent\emph{Singular matrices.}
Suppose that $0\ne A\in K$ is singular.
For a unit vector $x\in\ker A$, we have $AP_x=0$ and $A^*x\perp x$.
Write
\[
 A=N+B,\qquad N=P_xA=x\otimes(A^*x),\qquad
 B=(I-P_x)A(I-P_x).
\]
If $N\ne0$, it is a rank-one nilpotent.
Since $B$ is supported on $x^\perp$, compression gives $B\sorth N$.
For every $X\in M_n(\C)$,
\[
 2w(N+BX)\ge\|(N+BX)^*x\|=\|A^*x\|=2w(N),
\]
so $N\morth B$.
Their images sum to zero and are mutually strongly orthogonal;
hence both vanish, so $N\in K$.
By the first step, $N=0$, and therefore $A^*x=0$ for every $x\in\ker A$.
Equality of dimensions gives $\ker A=\ker A^*$.
The space $L=(\ker A)^\perp$ reduces $A$, and $A|_L$ is invertible.

Fix a unit vector $x\in\ker A$ and a unit maximizing vector
$m\in L$ for $A$.
Put $r=w(A)>0$, $N=x\otimes m$, and $A_t=A+tN$ for $t>0$.
We construct $N_t\morth A_t$ with $N_t\to N$.
For a unit vector $u=ax+z$, where $z\perp x$,
\[
 |\langle u,A_tu\rangle|
 =|\langle z,Az\rangle+t\overline a\,\langle m,z\rangle|
 \le r\|z\|^2+t|a|\|z\|.
\]
For fixed $|a|$ and $\|z\|$, equality is attained by taking
$z=bm$, $b\ge0$, and choosing the phase of $a$ so that the two
terms have the same argument.
We may therefore choose a unit maximizing vector
$m_t=a_tx+b_tm$, with $b_t\ge0$. Since
\[
 r=|\langle m,A_tm\rangle|
 \le w(A_t)=rb_t^2+t|a_t|b_t
 \le rb_t^2+t,
\]
we have $|a_t|^2=1-b_t^2\le t/r$.
Hence $a_t\to0$ and $b_t\to1$, so $m_t\to m$ as $t\downarrow0$.

Set
\[
 k_t=x-t(A^*|_L)^{-1}m.
\]
Then $A_t^*k_t=0$ and $k_t\to x$.
Let $u_t$ be the normalized projection of $k_t$ onto $m_t^\perp$.
For sufficiently small $t$ this is well defined and $u_t\to x$.
Writing
\[
 k_t=\alpha_tu_t+\beta_tm_t,
 \qquad \alpha_t\to1,\quad\beta_t\to0,
\]
we obtain
\[
 A_t^*u_t=-\frac{\beta_t}{\alpha_t}A_t^*m_t,
 \qquad \left|\frac{\beta_t}{\alpha_t}\right|\le1
\]
for sufficiently small $t$.
Put $N_t=u_t\otimes m_t$.
Lemma~\ref{lem:nilpotent} gives $N_t\sorth A_t$,
and $N_t^*m_t=0$ gives $A_t\sorth N_t$ by
Lemma~\ref{lem:witness}.
Thus $N_t\morth A_t$ and $N_t\to N$.
Since $\Phi(A_t)=t\Phi(N)$, preservation and scalar homogeneity
give $\Phi(N_t)\morth\Phi(N)$.
Continuity of $\Phi$ and $w$ then yields
$\Phi(N)\morth\Phi(N)$, so $\Phi(N)=0$.
This contradicts the exclusion of nonzero rank-one nilpotents from $K$.
Hence every nonzero element of $K$ is invertible.

\smallskip
\noindent\emph{Preservation of singularity.}
Let $A\ne0$ be singular, and choose a unit maximizing vector $m$ for $A$.
By Lemma~\ref{lem:hyperplane-span}, either $A\in\mathcal V(A)$ or
$\mathcal H_m\subseteq\mathcal V(A)$.
Since $A\ne0$ and $\mathcal H_m\ne\{0\}$, we have
$\mathcal V(A)\ne\{0\}$, so we may choose $0\ne B\in\mathcal N(A)$.
If $B$ were invertible, taking $X=-B^{-1}A$ in the defining inequality
for $A\sorth B$ would give $A=0$, a contradiction.
Thus $B$ is singular, and the preceding argument gives $\Phi(B)\ne0$.
Preservation yields $\Phi(B)\sorth\Phi(A)$.
If $\Phi(A)$ were invertible, taking $X=-\Phi(A)^{-1}\Phi(B)$ in
this inequality would give $\Phi(B)=0$, again a contradiction.
Hence $\Phi(A)$ is singular. The case $A=0$ is immediate.
\end{proof}

\begin{lemma}\label{lem:singular-range}
Let $T:M_n(\C)\to M_n(\C)$ be real linear and map singular
matrices to singular matrices. If $\ker T\ne\{0\}$, then
$\mathcal S=\Span_{\C}T(M_n(\C))$ consists of singular matrices, and
\[
 \dim_{\R}T(M_n(\C))\le2n(n-1).
\]
\end{lemma}
\begin{proof}
Put $p(A)=\det T(A)$, a homogeneous polynomial of degree $n$ in the
real and imaginary parts of the entries of $A$.
It vanishes whenever a row of $A$ is zero, so each nonzero monomial
contains exactly one real coordinate from each row.
Thus $p$ is real multilinear in the rows.
Adding a complex multiple of one row to another leaves $p$ unchanged:
by additivity in that row, the extra term has two complex linearly
dependent rows and hence vanishes.
Since these elementary row operations generate $\mathrm{SL}(n,\C)$,
\[
 p(SA)=p(A)\qquad(S\in \mathrm{SL}(n,\C)).
\]

Take $0\ne V\in\ker T$.
For any invertible $A$, choose $x$ with $Vx\ne0$ and
$S\in \mathrm{SL}(n,\C)$ such that $SAx=Vx$.
Such an $S$ exists because $n\ge2$: extend $Ax$ and $Vx$ to bases,
and rescale a second target basis vector to make the determinant one.
Then $SA-V$ is singular, while $T(SA-V)=T(SA)$, so
\[
 p(A)=p(SA)=p(SA-V)=0.
\]
Thus every matrix in $T(M_n(\C))$ is singular.
For any finite collection $T_1,\ldots,T_k$ in this real linear space,
the polynomial $\det(\sum_jz_jT_j)$ vanishes for all real
$z_1,\ldots,z_k$, and hence identically.
Therefore every matrix in $\mathcal S$ is singular.

Dieudonn\'e's theorem on spaces of singular matrices
\cite[Th\'eor\`eme~1]{D1949} gives
$\dim_{\C}\mathcal S\le n(n-1)$. Hence
\[
 \dim_{\R}T(M_n(\C))
 \le2\dim_{\C}\mathcal S\le2n(n-1).
\]
\end{proof}

\Needspace{10\baselineskip}
\begin{proposition}\label{prop:finite-bijectivity}
Every nonzero additive map $\Phi:M_n(\C)\to M_n(\C)$ preserving
$\morth$ in one direction is bijective.
\end{proposition}
\begin{proof}
By Lemmas~\ref{lem:finite-real-linearity} and~\ref{lem:singularity},
$\Phi$ is real linear and preserves singularity, and
$K:=\ker\Phi$ contains no nonzero singular matrix.
Put $E=\Phi(M_n(\C))$.
We use the Hilbert--Schmidt norm
$\|T\|_{\mathrm{HS}}=(\tr(T^*T))^{1/2}$.
Suppose that $K\ne\{0\}$.
Lemma~\ref{lem:singular-range} gives
$\dim_{\R}E\le2n(n-1)$.
For every unit vector $m$, the space $\mathcal H_m$ consists of
singular matrices and has real dimension $2n(n-1)$.
Since $K$ contains no nonzero singular matrix, the restriction of
$\Phi$ to $\mathcal H_m$ is injective. Thus
\[
 2n(n-1)=\dim_{\R}\mathcal H_m
 \le\dim_{\R}E\le2n(n-1),
\]
and hence
\begin{equation}\label{eq:finite-hyperplane-image}
 \Phi(\mathcal H_m)=E\qquad(\|m\|=1).
\end{equation}

Choose $T\in E$ with $\|T\|_{\mathrm{HS}}=1$ such that $w(T)$ is
maximal on this sphere.
For $D\in E$ with $D\morth T$ and every $s\in\R$,
\[
 w(T)\le w(T+sD)\le w(T)\|T+sD\|_{\mathrm{HS}}.
\]
Expanding the square of the last norm for positive and negative $s$
gives $\operatorname{Re}\tr(T^*D)=0$.
By~\eqref{eq:finite-hyperplane-image}, there is a nonzero singular
matrix $A$ with $\Phi(A)=T$.
Preservation and real linearity now imply
\[
 \Phi(\mathcal V(A))
 \subseteq\{D\in E:\operatorname{Re}\tr(T^*D)=0\}.
\]
Choose a unit maximizing vector $m$ for $A$.
Lemma~\ref{lem:hyperplane-span} gives either $A\in\mathcal V(A)$
or $\mathcal H_m\subseteq\mathcal V(A)$.
In either case, using~\eqref{eq:finite-hyperplane-image} in the
second, we obtain $T\in\Phi(\mathcal V(A))$.
This contradicts the preceding inclusion and
$\|T\|_{\mathrm{HS}}=1$.
Thus $K=\{0\}$.
Real linearity and finite dimensionality now make $\Phi$ bijective.
\end{proof}

\begin{proof}[Proof of Theorem~\ref{thm:full} in finite dimension]
Let $\Phi$ be a nonzero additive map preserving $\morth$ in one
direction.
By Lemma~\ref{lem:finite-real-linearity} and
Proposition~\ref{prop:finite-bijectivity}, it is a real-linear
bijection, and Lemma~\ref{lem:singularity} gives preservation of
singularity.
The additive singularity-preserver theorem
\cite[Theorem~1.4]{KS2026}, with one matrix block, gives invertible
$X,Y\in M_n(\C)$ and a field endomorphism $\sigma:\C\to\C$ such that
\begin{equation}\label{eq:finite-semilinear}
 \Phi(A)=XA^\sigma Y
 \quad\text{or}\quad
 \Phi(A)=X(A^\sigma)^tY,
\end{equation}
where $A^\sigma$ denotes entrywise application of $\sigma$.
For $t\in\R$, real linearity gives
\[
 \sigma(t)XE_{11}Y=\Phi(tE_{11})=t\Phi(E_{11})=tXE_{11}Y.
\]
Since $XE_{11}Y\ne0$, we have $\sigma(t)=t$ for every real $t$.
Together with $\sigma(i)=\pm i$, this shows that $\sigma$ is either
the identity or complex conjugation.
In particular, $\Phi$ preserves rank.

Let $N$ be a nonzero rank-one nilpotent.
Real linearity, bijectivity, and Lemma~\ref{lem:span} give
\[
 M_n(\C)=\Phi(\mathcal V(N))\subseteq\mathcal V(\Phi(N)).
\]
The rank-one matrix $\Phi(N)$ is therefore nilpotent, again by
Lemma~\ref{lem:span}.
Put $K=YX$.
For every real rank-one nilpotent $N$, we have $N^\sigma=N$,
so the nilpotency of $\Phi(N)$ gives $\tr(NK)=0$ in the first
form of~\eqref{eq:finite-semilinear} and $\tr(N^tK)=0$ in the second.
Taking $N=E_{ij}$ for all $i\ne j$ shows that $K$ is diagonal.
Taking
\[
 N=(e_i+e_j)\otimes(e_i-e_j)\qquad(i\ne j)
\]
then gives $K_{ii}=K_{jj}$ in either case.
Hence $YX=cI$ for some $c\ne0$.
After dividing $\Phi$ by $c$, we obtain a unital map
\[
 \Psi(A)=XA^\sigma X^{-1}
 \quad\text{or}\quad
 \Psi(A)=X(A^\sigma)^tX^{-1}.
\]

For every rank-one orthogonal projection $P$, we have $P\morth I-P$.
Its image $Q=\Psi(P)$ is a rank-one idempotent satisfying
$Q\sorth I-Q$.
Lemma~\ref{lem:idempotent} shows that $Q$ is an orthogonal projection.
Apply this to the real projections
\[
 P=E_{ii}
 \quad\text{and}\quad
 P=P_{ij}:=\tfrac12(E_{ii}+E_{ij}+E_{ji}+E_{jj}),\qquad i\ne j.
\]
Since $P^\sigma=P^t=P$, the matrix $XPX^{-1}$ is self-adjoint,
so $X^*X$ commutes with $P$.
Commutation with all $E_{ii}$ makes $X^*X$ diagonal, and
commutation with all $P_{ij}$ makes its diagonal entries equal.
Thus $X^*X=rI$ for some $r>0$, and $X=\sqrt r\,U$ with $U$ unitary.

The transpose alternatives are impossible.
Indeed, $E_{12}\morth E_{22}$, whereas
$E_{22}\not\sorth E_{21}$, because
\[
 E_{22}+E_{21}(-E_{12})=0.
\]
The same real pair excludes the transpose followed by entrywise
conjugation.
We conclude that
\[
 \Phi(A)=cUAU^*\quad\text{or}\quad
 \Phi(A)=cU\overline A\,U^*.
\]
If $J$ is the standard conjugation on $\C^n$, the second form equals
$c(UJ)A(UJ)^*$, where $UJ$ is antiunitary.
Thus both forms have the representation~\eqref{eq:full-standard}.
Conversely, they preserve $\sorth$, and hence $\morth$, in both
directions by Lemma~\ref{lem:standard}.
\end{proof}

\section{The infinite-dimensional case of Theorem~\texorpdfstring{\ref{thm:full}}{\getrefnumber{thm:full}}}\label{sec:infinite}

Throughout this section, $H$ is an infinite-dimensional complex Hilbert space.
We first show that the numerical radius is preserved up to a constant
factor, without assuming surjectivity. Write $\mathcal F(H)$ for
the space of finite-rank operators on $H$.

\begin{proposition}\label{prop:radius-similarity}
Let $\Phi:\BH\to\BH$ be a nonzero additive map satisfying
\begin{equation}\label{inf-forward}
 A\morth B\quad\Longrightarrow\quad\Phi(A)\morth\Phi(B)
 \qquad(A,B\in\BH).
\end{equation}
Then $\Phi$ is bounded and real linear, and there is a constant $\kappa>0$
such that
\begin{equation}\label{inf-radius-similarity}
 w(\Phi(A))=\kappa w(A)\qquad(A\in\BH).
\end{equation}
In particular, $\Phi$ is injective.
\end{proposition}
\begin{proof}
Put $q(A)=w(\Phi(A))$.

\smallskip
\noindent\textbf{Step 1: automatic continuity.}
Fix orthonormal vectors $x,y\in H$, and put $P=P_x$ and $Q=P_y$.
If $D$ is supported on $\{x,y\}^{\perp}$ and $w(D)\le1$, then
\[
 P+D\morth Q+D.
\]
Indeed, both numerical radii equal $1$, and compression to $\C x$ and
$\C y$ gives the two defining inequalities.  By preservation, taking
$X=-I$ in the inequality for the images and using additivity gives
\begin{equation}\label{inf-corner-bound}
 q(P+D)\le q(P-Q),\qquad q(D)\le q(P-Q)+q(P).
\end{equation}
Choose an orthogonal decomposition $H=H_1\oplus H_2\oplus H_3$ with
each $H_i$ infinite dimensional, and let $E_i$ be the associated
projections.  Fix two orthonormal vectors in each $H_i$.  For $\|A\|\le1$,
write
\[
 A=\sum_{i,j=1}^{3}E_iAE_j.
\]
Each summand $E_iAE_j$ has norm at most $1$ and is supported on
$H_i+H_j$.  For each pair $(i,j)$, choose $k\notin\{i,j\}$ and
apply~\eqref{inf-corner-bound} with the fixed orthonormal pair in $H_k$.
This bounds $q(E_iAE_j)$ by a constant independent of $A$, so subadditivity
shows that $q$ is bounded on the operator-norm unit ball.
Since $\|\Phi(A)\|\le2q(A)$, $\Phi$ is bounded on the unit ball and hence
continuous and real linear.  Thus $q$ is a continuous real seminorm on $\BH$.

\smallskip
\noindent\textbf{Step 2: orthogonal finite-rank perturbations.}
Let $A,F\in\mathcal F(H)$ have orthogonal supports, and suppose that
$w(F)\le r:=w(A)$.  We claim that
\begin{equation}\label{inf-halving-identity}
 q(A+F)=q(A).
\end{equation}
To prove $q(A+F)\le q(A)$, we successively halve the numerical radius
of the perturbation. An auxiliary rank-one term in a new orthogonal
direction will provide the mutual orthogonality needed to compare
the values of $q$ before and after each step.
Set $F_0=F$.  Given $F_j$, put $t_j=w(F_j)$, choose a rank-one
orthogonal projection $P_j$ supported orthogonally to $A$ and $F_j$,
and set
\begin{equation}\label{inf-halving-recursion}
 G_j=F_j+t_jP_j,\qquad F_{j+1}=\frac{F_j-t_jP_j}{2}.
\end{equation}
Such a projection exists because the supports involved are finite
dimensional.  Each $F_j$ is finite rank and has support orthogonal to
that of $A$.  The direct-sum identity in~\eqref{eq:basic-radius} gives
\[
 w(G_j)=t_j,\qquad w(F_{j+1})=\frac{t_j}{2},\qquad w(A+F_j)=r.
\]
Moreover,
\[
 A+F_j\morth G_j.
\]
Indeed, for every $X\in\BH$, compressing $A+F_j+G_jX$ to the support
of $A$ gives $A$, while compressing $G_j+(A+F_j)X$ to $P_jH$ gives
$t_jP_j$.  These compressions have numerical radii $r$ and $t_j$,
respectively.
Preservation and the choice $X=-I/2$ therefore give
\begin{align*}
 q(A+F_j)
 &\le w\left(\Phi(A+F_j)-\tfrac12\Phi(G_j)\right)\\
 &=q\left(A+F_j-\tfrac12G_j\right)=q(A+F_{j+1}).
\end{align*}
Since $w(F_j)=2^{-j}w(F)$ and $\|F_j\|\le2w(F_j)\to0$, continuity gives
$q(A+F)\le q(A)$.  For the reverse inequality, $A\morth F$ and preservation, with
$X=I$, give $q(A)\le q(A+F)$.  This proves~\eqref{inf-halving-identity}.

\smallskip
\noindent\textbf{Step 3: one constant on all finite-rank operators.}
If $A,B\in\mathcal F(H)$ have orthogonal supports and $w(A)=w(B)$,
then~\eqref{inf-halving-identity}, applied in both orders, gives
\[
 q(A)=q(A+B)=q(B).
\]
Fix a rank-one orthogonal projection $P_0$, and put $\kappa=q(P_0)\ge0$.
For any nonzero $F\in\mathcal F(H)$, choose a rank-one orthogonal
projection $P$ supported orthogonally to both $F$ and $P_0$.
The preceding equality gives
\[
 q(F/w(F))=q(P)=q(P_0)=\kappa.
\]
Real homogeneity now yields
\begin{equation}\label{inf-finite-radius}
 q(F)=\kappa w(F)\qquad(F\in\mathcal F(H)).
\end{equation}

\smallskip
\noindent\textbf{Step 4: a global upper bound.}
Suppose first that $D$ is supported on the orthogonal complement of two
orthonormal vectors $x,y$.  Put $P=P_x$, $Q=P_y$, and $r=w(D)$.
The case $r=0$ is immediate.  If $r>0$, compression to $\C x$ and
$\C y$ gives
\[
 D+rP\morth D+rQ,\qquad D-rP\morth D+rQ.
\]
Applying preservation, taking $X=-I$, and using~\eqref{inf-finite-radius},
we obtain
\[
 q(D+rP)\le q(rP-rQ)=\kappa r,\qquad
 q(D-rP)\le q(-rP-rQ)=\kappa r.
\]
Averaging these two inequalities yields
\begin{equation}\label{inf-corner-upper}
 q(D)\le\kappa w(D).
\end{equation}

Now let $A\in\BH$ be arbitrary.  Choose pairwise orthogonal rank-two
projections $Q_1,Q_2,\ldots$, and put
\[
 D_j=(I-Q_j)A(I-Q_j),\qquad Q^{(m)}=\sum_{j=1}^{m}Q_j.
\]
By~\eqref{inf-corner-upper} and compression,
$q(D_j)\le\kappa w(D_j)\le\kappa w(A)$.  Also,
\begin{equation}\label{inf-compression-average}
 A-\frac1m\sum_{j=1}^{m}D_j
 =\frac1m\left(Q^{(m)}A+AQ^{(m)}
              -\sum_{j=1}^{m}Q_jAQ_j\right).
\end{equation}
The block-diagonal operator $\sum_{j=1}^{m}Q_jAQ_j$ has norm at most
$\|A\|$, as do the other two terms in parentheses.
Thus the norm of~\eqref{inf-compression-average} is at most $3\|A\|/m$,
so the averages $m^{-1}\sum_{j=1}^{m}D_j$ converge to $A$ in norm.
Continuity and the triangle inequality for $q$ imply
\begin{equation}\label{inf-global-upper}
 q(A)\le\kappa w(A)\qquad(A\in\BH).
\end{equation}
If $\kappa=0$, this would force $\Phi=0$.  Hence $\kappa>0$.

\smallskip
\noindent\textbf{Step 5: the matching lower bound.}
For a unit vector $x$ and an operator $A$, we first verify
\begin{equation}\label{inf-projection-derivative}
 w(P_x+tA)=1+t\operatorname{Re}\langle x,Ax\rangle
              +O(t^2\|A\|^2)
 \qquad(t\in\R,\ t\to0).
\end{equation}
Let $a=\langle x,Ax\rangle$ and $M=\|A\|$.  Write a unit vector as
$z=\alpha x+y$, where $y\perp x$, and set $s=\|y\|$.  Then
\begin{align*}
 |\langle z,(P_x+tA)z\rangle|
 &\le(1-s^2)|1+ta|+2|t|Ms+|t|Ms^2\\
 &=|1+ta|-bs^2+2|t|Ms,\qquad b=|1+ta|-|t|M.
\end{align*}
If $|t|M\le1/4$, then $b\ge1/2$.  Maximizing this quadratic bound in $s$
and taking $z=x$ for the lower bound, we obtain
\[
 |1+ta|\le w(P_x+tA)\le |1+ta|+2t^2M^2.
\]
This proves~\eqref{inf-projection-derivative}.

Fix a unit vector $x$ and $\theta\in\R$.  For $t>0$,
\eqref{inf-finite-radius}, the triangle inequality, and
\eqref{inf-global-upper} give
\begin{align*}
 \kappa=q(e^{i\theta}P_x)
 &\le q(e^{i\theta}P_x-tA)+tq(A)\\
 &\le\kappa w(e^{i\theta}P_x-tA)+tq(A).
\end{align*}
Consequently,
\[
 q(A)\ge\kappa\frac{1-w(e^{i\theta}P_x-tA)}{t}.
\]
Since $w(e^{i\theta}P_x-tA)=w(P_x-te^{-i\theta}A)$,
applying~\eqref{inf-projection-derivative} to $-e^{-i\theta}A$ and letting
$t\downarrow0$ gives
\[
 q(A)\ge\kappa\operatorname{Re}
             \bigl(e^{-i\theta}\langle x,Ax\rangle\bigr).
\]
Taking the supremum over $x$ and $\theta$ yields $q(A)\ge\kappa w(A)$.
Together with~\eqref{inf-global-upper}, this proves
\eqref{inf-radius-similarity}.  Injectivity follows because $w$ is a norm.
\end{proof}

\begin{proof}[Proof of Theorem~\ref{thm:full} in infinite dimension]
The converse follows from Lemma~\ref{lem:standard}.  For necessity,
let $\Phi$ be an additive surjection satisfying the one-direction
preservation condition.  By Proposition~\ref{prop:radius-similarity},
there is $\kappa>0$ such that $w(\Phi(A))=\kappa w(A)$ for every $A$.
The surjective map $\Psi=\kappa^{-1}\Phi$ therefore satisfies
\[
 w\bigl(\Psi(A)-\Psi(B)\bigr)=w(A-B)\qquad(A,B\in\BH).
\]
Since $\Psi(0)=0$, the numerical-radius distance-isometry theorem of
Bai and Hou~\cite[Theorem~$1'$]{BH2004} yields a unimodular scalar $\mu$
and a unitary or antiunitary $U$ such that either
\[
 \Psi(A)=\mu UAU^*\qquad(A\in\BH),
\]
or
\begin{equation}\label{inf-adjoint-alternative}
 \Psi(A)=\mu UA^*U^*\qquad(A\in\BH).
\end{equation}

If~\eqref{inf-adjoint-alternative} held, Lemma~\ref{lem:standard}
would imply that the adjoint map preserves $\morth$.
Choose orthonormal $e_1,e_2$ and put $P=e_1\otimes e_1$ and
$N=e_2\otimes e_1$.  Lemmas~\ref{lem:witness} and~\ref{lem:nilpotent}
give $P\morth N$, whereas
\[
 P+N^*(-N)=P-N^*N=0
\]
shows that $P\not\sorth N^*$.  Thus~\eqref{inf-adjoint-alternative}
is impossible.  Taking $c=\kappa\mu$ gives
\[
 \Phi(A)=cUAU^*\qquad(A\in\BH),
\]
as required.
\end{proof}

\section{The self-adjoint case}\label{sec:selfadjoint}

Throughout this section, $\dim H\ge2$. We use $w(A)=\|A\|$ for
$A\in\BSH$. In particular, the definition gives
\begin{equation}\label{eq:selfadjoint-sum-bound}
 A\morth B
 \quad\Longrightarrow\quad
 \max\{\|A\|,\|B\|\}\le\|A\pm B\|
 \qquad(A,B\in\BSH).
\end{equation}
Also, $AB=0$ implies $A\morth B$ for $A,B\in\BSH$.

\begin{lemma}\label{lem:selfadjoint-continuity}
Every additive map $\Phi:\BSH\to\BSH$ preserving $\morth$ in one
direction is continuous and real linear.
\end{lemma}
\begin{proof}
The finite-dimensional case follows from
Lemma~\ref{lem:finite-real-linearity}.
Suppose now that $H$ is infinite dimensional.
Fix orthonormal vectors $x,y\in H$. If $A\in\BSH$ satisfies
$Ax=Ay=0$ and $\|A\|\le1$, then
$P_x+A\morth P_y+A$.
Put $C=\Phi(P_x)$ and $D=\Phi(P_y)$.
Applying~\eqref{eq:selfadjoint-sum-bound} to the images gives
\[
 \|C+\Phi(A)\|\le\|C-D\|,
 \qquad
 \|\Phi(A)\|\le\|C\|+\|C-D\|.
\]
Thus $\Phi$ is bounded on the unit ball of
$\{A\in\BSH:Ax=Ay=0\}$.

Write $H=H_1\oplus H_2\oplus H_3$, with each $H_i$ infinite
dimensional, and let $Q_i$ be the corresponding orthogonal projections.
Every $A\in\BSH$ is the sum
\[
 A=\sum_{i=1}^3Q_iAQ_i+
   \sum_{1\le i<j\le3}(Q_iAQ_j+Q_jAQ_i).
\]
Each of these six self-adjoint summands has norm at most $\|A\|$
and is supported on at most two of the subspaces $H_i$.
Choose two fixed orthonormal vectors in each $H_i$.
The preceding estimate applies to each summand using the pair in an
omitted subspace. Hence $\Phi$ is bounded on the unit ball, and therefore
continuous and real linear.
\end{proof}

\Needspace{6\baselineskip}
\begin{lemma}\label{lem:selfadjoint-injectivity}
Every nonzero additive map $\Phi:\BSH\to\BSH$ preserving $\morth$
in one direction is injective.
\end{lemma}
\begin{proof}
By Lemma~\ref{lem:selfadjoint-continuity}, $\Phi$ is continuous and real linear.
Suppose first that $\dim H=2$.
In this case, $R\morth S$ is equivalent to $RS=0$ for $R,S\in\BSH$.
Indeed, if $R\morth S$ and neither matrix is zero, both are singular,
and hence nonzero real multiples of rank-one orthogonal projections.
Equation~\eqref{eq:rankone-kernel} then shows that their ranges are orthogonal.

Put $C=\Phi(I)$. For each rank-one orthogonal projection $P$, set
$R=\Phi(P)$ and $S=\Phi(I-P)$. Then $RS=SR=0$ and $C=R+S$.
If $C=0$, then $R^2=0$ and hence $R=0$ for every $P$.
The spectral decomposition would give $\Phi=0$, a contradiction.
Thus $C\ne0$. Every nonzero eigenvalue of $R$ is an eigenvalue of $C$,
since $S$ vanishes on $\ran R$. Consequently,
\[
 \Phi(P)=0\quad\text{or}\quad\|\Phi(P)\|\ge\delta,
 \qquad
 \delta=\min\{|\lambda|:\lambda\in\sigma(C)\setminus\{0\}\}>0.
\]
The rank-one orthogonal projections form a connected set, and
$P\mapsto\|\Phi(P)\|$ is continuous. Since $\Phi\ne0$ and these
projections span $\BSH$ over $\R$, this gap implies $\Phi(P)\ne0$
for every such $P$.
If $\Phi(aP+b(I-P))=0$ for $a,b\in\R$, then $aR+bS=0$.
The ranges of $R$ and $S$ are orthogonal and both matrices are nonzero,
so $a=b=0$. The spectral decomposition now gives injectivity.

Assume now that $\dim H\ge3$.
We first show that $\Phi(P)\ne0$ for every rank-one orthogonal
projection $P$.
Suppose that $P_1,P_2,P_3$ are mutually orthogonal rank-one projections
with $\Phi(P_1)=0$, and put $C=\Phi(P_2)$ and $D=\Phi(P_3)$.
By compression, we have
\[
 P_1+P_2\morth P_2\pm P_3,
 \qquad
 P_1+P_3\morth P_3\pm P_2.
\]
These relations give $C\morth C\pm D$ and $D\morth D\pm C$.
Taking $X=-I/2$ in $C\sorth C+D$, and $X=-I$ in
$C-D\sorth C$, yields
\[
 2\|C\|\le\|C-D\|\le\|D\|.
\]
Interchanging $C$ and $D$ gives $2\|D\|\le\|C\|$, so $C=D=0$.
Thus $\Phi$ vanishes on all rank-one projections orthogonal to $P_1$.
Since $\dim H\ge3$, for every rank-one projection $P$ there is a rank-one
projection $Q$ orthogonal to both $P_1$ and $P$.
The preceding argument gives $\Phi(Q)=0$ and then $\Phi(P)=0$.
Hence $\Phi$ vanishes on every finite-rank self-adjoint operator.

For any $T\in\BSH$, choose orthonormal $x,y$ and put
$Q=I-P_x-P_y$ and $T_0=QTQ$.
Since $T-T_0$ has finite rank, $\Phi(T)=\Phi(T_0)$.
For $t>\|T_0\|$ we have
\[
 tP_x+T_0\morth tP_y+T_0.
\]
Both images equal $\Phi(T)$, so
\eqref{eq:self-orthogonality} gives $\Phi(T)=0$.
This contradicts $\Phi\ne0$ and proves the claim.

Now suppose that $0\ne A\in\ker\Phi$.
For every reducing orthogonal projection $P$ of $A$, the operators
$AP$ and $A(I-P)$ have orthogonal supports.
Since $\Phi(AP)\morth\Phi(A(I-P))$ and their sum is zero,
we obtain $\Phi(AP)=0$.
If $A$ has a nonzero eigenvalue, a corresponding rank-one projection
therefore belongs to $\ker\Phi$, contradicting the claim.
Otherwise, $\sigma(A)\setminus\{0\}$ is infinite: if it were finite,
each of its points would be isolated in $\sigma(A)$ and hence an
eigenvalue of $A$. We may therefore choose three disjoint closed
intervals avoiding zero whose spectral projections $E_1,E_2,E_3$
for $A$ are nonzero.
Set $A_i=AE_i$ for $i=1,2$, choose a unit vector $x\in E_3H$, and take
\[
 0<t\le\min\{\|A_1\|,\|A_2\|\},\qquad F=tP_x.
\]
We have $\Phi(A_1)=\Phi(A_2)=0$.
Since $A_i$ and $F$ have orthogonal supports and $\|F\|=t\le\|A_i\|$,
the direct-sum identity gives
\[
 w(A_i+F)=\max\{\|A_i\|,t\}=\|A_i\|=w(A_i)
 \qquad(i=1,2).
\]
Compression to $E_1H$ and $E_2H$ now gives
$A_1+F\morth A_2+F$.
Both images equal $\Phi(F)$, so $\Phi(P_x)=0$, again a contradiction.
Thus $\Phi$ is injective.
\end{proof}

\begin{proof}[Proof of Theorem~\ref{thm:selfadjoint}]
By Lemmas~\ref{lem:selfadjoint-continuity} and~\ref{lem:selfadjoint-injectivity},
$\Phi$ is continuous, real linear, and injective.
It is therefore surjective in finite dimension, while surjectivity is
assumed in infinite dimension.

Let $0\ne A\in\BSH$ be noninvertible, and put
\[
 r=\|A\|,\qquad B=I-\frac{|A|}{r}.
\]
Since $0\in\sigma(A)$, we have $\|B\|=1$.
There are unit vectors $u_j,v_j$ and
$\lambda\in\{r,-r\}\cap\sigma(A)$ such that
\[
 (A-\lambda I)u_j\longrightarrow0,
 \qquad Av_j\longrightarrow0.
\]
Then $Bu_j\to0$ and $Bv_j-v_j\to0$, so
Lemma~\ref{lem:witness} gives $A\morth B$.
Since $\Phi(B)\ne0$, the relation
$\Phi(B)\sorth\Phi(A)$ implies that $\Phi(A)$ is noninvertible.
Thus $\Phi$ preserves noninvertibility, and $\Phi^{-1}$ maps invertible
self-adjoint operators to invertible operators.

Apply the theorem of Li, Rodman, and \v{S}emrl
\cite[Theorem~6]{LRS2003} to the real-linear bijection $\Phi^{-1}$.
Inverting the resulting representation gives an invertible $T\in\BH$
and $\varepsilon\in\{1,-1\}$ such that
\[
 \Phi(A)=\varepsilon TAT^*
 \quad\text{or}\quad
 \Phi(A)=\varepsilon TA^{\mathrm t}T^*
 \qquad(A\in\BSH).
\]
Here transpose is taken relative to a fixed conjugation $J$, so
$A^{\mathrm t}=JAJ$ on $\BSH$.
In either case, applying $\Phi$ to orthogonal rank-one projections
and using~\eqref{eq:rankone-kernel} shows that $T$ preserves vector
orthogonality.
Hence $T^*T=\rho I$ for some $\rho>0$, and $T=\sqrt\rho\,V$
with $V$ unitary.
Taking $c=\varepsilon\rho$ and $U=V$ or $VJ$ gives the asserted form.
The converse follows from Lemma~\ref{lem:standard}.
\end{proof}

\section{Sharpness of the hypotheses}\label{sec:remarks}

\begin{example}\label{ex:dimension-one}
The restriction $\dim H\ge2$ in Theorems~\ref{thm:full}
and~\ref{thm:selfadjoint} is necessary. On $\C$ and
$\mathcal B_s(\C)=\R$, the relation $a\morth b$ holds exactly when
$a=0$ or $b=0$, so every additive map preserves it in one direction.
The additive bijection $a\mapsto2a+\overline a$ on $\C$ is neither
$a\mapsto ca$ nor $a\mapsto c\overline a$. On $\R$, any non-real-linear
additive bijection gives the corresponding counterexample.
\end{example}

\Needspace{8\baselineskip}
\begin{example}\label{ex:embedding}
In infinite dimension, surjectivity cannot be replaced by injectivity in
Theorems~\ref{thm:full} and~\ref{thm:selfadjoint}.
The linear embedding
\[
 \iota:\mathcal B(H)\longrightarrow\mathcal B(H\oplus H),
 \qquad \iota(A)=A\oplus0,
\]
satisfies $w(\iota(A))=w(A)$ and preserves $\sorth$ in both directions.
Indeed, if $A\sorth B$, then compression to the first summand gives,
for every block operator $X$,
\[
 w\bigl((A\oplus0)+(B\oplus0)X\bigr)
 \ge w(A+BX_{11})\ge w(A).
\]
Conversely, choosing $X=X_{11}\oplus0$ gives the original defining
inequality. Hence $\iota$ also preserves the mutual relation in both
directions. As $H\cong H\oplus H$ in infinite dimension, this gives an
injective linear self-map of $\BH$ after a unitary identification.
Its value at $I$ is a proper projection, so it does not have the form
in Theorem~\ref{thm:full}.
Restricting the same embedding to self-adjoint operators gives the
corresponding example for Theorem~\ref{thm:selfadjoint}.
\end{example}

\begin{remark}\label{rem:nonadditive}
Additivity cannot be omitted from Theorems~\ref{thm:full}
and~\ref{thm:selfadjoint}. Every surjective operator $A\in\BH$ has a
bounded right inverse $Q$. For $B\ne0$, the identity
\[
 B+A(-QB)=0
\]
shows that $B\not\sorth A$. Thus $\mathcal N(A)=\{0\}$.
Consequently, any permutation of the surjective operators, extended by
the identity on all nonsurjective operators, preserves $\morth$ in both
directions. In particular, the bijection $\Theta$ interchanging $I$
and $2I$ and fixing every other operator is not additive, since
$\Theta(2I)=I\ne4I=2\Theta(I)$.
It restricts to a bijection of $\BSH$ with the same preservation
property. This construction works in every dimension.

In infinite dimension, interchanging $I$ with a surjective noninjective
coisometry and fixing every other operator gives a bijection of $\BH$
preserving $\morth$ in both directions. Thus the relation $\morth$
alone does not distinguish invertible operators from noninjective
surjective operators.
\end{remark}

\end{document}